\documentclass[review]{elsarticle}
\usepackage{xcolor}
\usepackage{amsmath}
\usepackage{amsfonts}
\usepackage{amssymb}
\usepackage{lineno,hyperref}
\usepackage{array}
\usepackage{graphicx}
\usepackage{epstopdf}
\usepackage{pdfpages}
\newtheorem{theorem}{Theorem}

\newtheorem{definition}[theorem]{Definition}
\newtheorem{example}[theorem]{Example}

\newtheorem{lemma}[theorem]{Lemma}

\newtheorem{proposition}[theorem]{Proposition}
\newtheorem{remark}[theorem]{Remark}

\newenvironment{proof}[1][Proof]{\noindent\textbf{#1.} }{\ \rule{0.5em}{0.5em}}
\modulolinenumbers[200]

\journal{  Journal }

\begin{document}

\begin{frontmatter}

\title{Discrete delayed perturbation of Mittag-Leffler function and its application to linear fractional delayed difference system}

\author[mysecondaryaddress]{Mustafa AYDIN\corref{mycorrespondingauthor}}

\cortext[mycorrespondingauthor]{Corresponding author}
\ead{m.aydin@yyu.edu.tr}

\author[mymainaddress,mymainaddress3]{Nazim I. MAHMUDOV}






\address[mysecondaryaddress]{Department of Medical Services and Techniques, Muradiye Vocational School,  Van Yuzuncu Yil University, Van , Turkey}
\address[mymainaddress]{Department of Mathematics, Eastern Mediterranean University, Famagusta 99628 T.R. North Cyprus, Turkey}
\address[mymainaddress3]{Research Center of Econophysics, Azerbaijan State University of Economics(UNEC), Istiqlaliyyat Str. 6., Baku 1001, Azerbaijan}

\begin{abstract}
The linear nonhomogeneous fractional difference system with constant coefficients is introduced. An explicit solution to the system is acquired by proposing a newly discrete retarded perturbation of the nabla Mittag-Leffer-type function containing such matrix equations that provide non-permutability. A couple of special cases obtained from our results are discussed.
\end{abstract}


\end{frontmatter}

\linenumbers

\section{ Introduction}
In the last two decades, fractional differential equations have been proved to be a powerful tool for lots of phenomena in the scientific world together with the usage of the areas such as mathematical physics\cite{k1}\cite{k2}, biology, computed tomography, diffusion, biophysics, signal process, engineering, electrochemistry, control theory\cite{k3}-\cite{k7}, etc.

Retardation is mostly related to chemical processes, economics, heredity in population, hydraulic and electrical networks. Generally a hallmark of the corresponding mathematical structure is that the rate of change of these courses depends on past history. Any differential equation consisting of at least one delay is known as a delayed differential equation. Combining a delayed differential equation with a fractional differential equation gives rise to highly realistic models for various systems having memory such as stabilization, automatic steering, control, and so on. Some articles about fractional delayed differential/difference equations can be found in the references\cite{k3k},\cite{ii1}-\cite{i6},\cite{i11},\cite{i3},\cite{c11},\cite{c16}.

As it is well-known, $z(t)=e^{Mt}z(0)$ is a solution to a linear system $z^{'}(t)=Mz(t), \  t\geq 0$, where $e^{Mt}$ is the exponential matrix function. Sometimes it is not so easy to find a solution to some linear systems like the following linear delayed system
\begin{equation}
\left\{
\begin{array}
[c]{l}%
 z^{'} \left( t\right) =Mz\left(  t\right)  +Nz\left(
t-r\right),\ \ t\in (0,T]  ,\\
\ \ \ \ z\left(  t\right)  =\phi\left(  t\right)  ,\ \ t\in[-r,0], \ \ r>0,\\
\end{array}
\right.  \label{x1}%
\end{equation}
where $M$ and $N$ are square matrices. Under the condition that $M$ and $N$ are commutative,  Khusainov and Shuklin\cite{i4} give a representation of solutions to system (\ref{x1}) by defining delayed exponential matrix function. Li and Wang\cite{i5} examine the fractional version of the same system with $M=\Theta$. Mahmudov\cite{i6} manages to acquire a solution to the delayed Caputo fractional differential system which is more general version of system (\ref{x1}) having a linear function  by proposing delay perturbation of Mittag-Leffler(DPML) matrix function.

 We observe the similar process for difference systems. Dibl\'ik  and Khusainov \cite{i7} \cite{iek7} focus on the linear discrete system having only one retardation and give its representation of solutions by introducing delayed discrete exponential matrix function $e_{h}^{Nk}$.  Dibl\'ik and Morávková\cite{i9}\cite{i10} find a representation of solutions of linear two-retarded discrete systems by extending the delayed discrete exponential matrix function to the delayed discrete exponential matrix function for two delays. Pospíšil\cite{i11} concentrates on the linear discrete multi-delayed system with commutative coefficient matrices and solves this system by using $Z$-transform. Mahmudov\cite{i12}  removes the commutativity in the same linear discrete multi-delayed system by proposing the delayed perturbation of discrete exponential matrix function $X_{h}^{M,N}\left(k\right)$ with the same method.

Jia et al. in the study\cite{i1}  prove that the M-L function $\mathbb{E}_{b,\beta,\beta-1}(k,\rho(a))$ $=\sum_{i=0}^{\infty}b^i H_{i\beta+\beta-1}\left(k,\rho(a)\right)$ is the unique solution to  the following the nabla Riemann Liouville fractional difference equation
\begin{equation}
\left\{
\begin{array}
[c]{l}%
\nabla_{\rho(a)}^\beta z \left( k\right) =bz\left(  k\right), \ \ \beta\in(0,1)     ,\ \ k\in\mathbb{N}_{a+1}  ,\\
\ \ \ \ z\left(  a\right)  =\frac{1}{1-b}>0.\\
\end{array}
\right.  \label{j1}%
\end{equation}
Jia et al. in the study\cite{i2}  prove that the M-L function $\mathbb{E}_{b,\beta,0}(k,a)=\sum_{i=0}^{\infty}b^i H_{i\beta}\left(k,a\right)$ is the unique solution to  the following the nabla Caputo fractional difference equation
\begin{equation}
\left\{
\begin{array}
[c]{l}%
^C\nabla_{a}^\beta z \left( k\right) =bz\left(  k\right), \ \ \beta\in(0,1)     ,\ \ k\in\mathbb{N}_{a+1}  ,\\
\ \ \ \ z\left(  a\right)  =1.\\
\end{array}
\right.  \label{j2}%
\end{equation}
Du and Lu in the work\cite{i3} consider the following nonhomogeneous delayed Riemann Liouville fractional difference system
\begin{equation}
\left\{
\begin{array}
[c]{l}%
\nabla_{-r}^\alpha z \left( k\right) =Nz\left(
k-r\right)  +\daleth\left(  k  \right)   ,\ \ k\in\mathbb{N}_1  ,\\
\ \ \ \ z\left(  k\right)  =\phi\left(  k\right)  ,\ \ k\in\mathbb{N}_{1-r}^0, \ \ r>0,\\
\end{array}
\right.  \label{DL3}%
\end{equation}
and present representation of solutions to system (\ref{DL3}) by proposing the discrete delayed M-L type matrix function $\mathbb{F}_r^{Nk^{\overline{\alpha}}}$, and examined its finite-time stability.

As stated in the reference\cite{yeniek1}, the Riemann Liouville has widespread usage in the real-wold problems. In the recent times, the Riemann Liouville fractional derivative presents an excellent tool to express anomalous diffusion, Levy flights, and so forth. The set of all functions which satisfy the definition of Riemann Liouville fractional derivative is bigger than the ones which fulfill the definition of the Caputo fractional derivative. Moreover, the RL fractional derivative is closer to the classical one than the Caputo one in terms of their features and similarities.

Motivated by the above-cited works and the aforementioned superiorities of both the fractional-order derivative and the one in the sense of Riemann Liouville, this current paper is dedicated to the exploration of the following linear retarded Riemann Liouville fractional difference system,
\begin{equation}
\left\{
\begin{array}
[c]{l}%
\nabla_{-r}^\alpha z \left( k\right) =Mz\left(  k\right)  +Nz\left(
k-r\right)  +\daleth\left(  k  \right)   ,\ \ k\in\mathbb{N}_1  ,\\
\ \ \ \ z\left(  k\right)  =\phi\left(  k\right)  ,\ \ k\in\mathbb{N}_{1-r}^0, \ \ r>0,\\
\end{array}
\right.  \label{sorkok2}%
\end{equation}
where $\nabla_{-r}^\alpha$ is the Riemann-Liouville fractional difference of order $0<\alpha<1$, $z:\mathbb{N}_1\rightarrow\mathbb{R}^n$, $\daleth:\mathbb{N}_1\rightarrow\mathbb{R}^n$ is a continuous function, $r \in \mathbb{N}_2$ is a retardation, $M, N \in \mathbb{R}^{n\times n}$ are constant coefficient matrices, $\phi:\mathbb{N}_{1-r}^0\rightarrow\mathbb{R}^n$ is an initial conditional function.

The major contributions are highlighted below.

\begin{itemize}
\item we introduce the nabla Riemann Liouville fractional delayed difference system with noncommutative coefficient matrices(see Section 2);
\item we newly define the discrete delay perturbation of the nabla M-L type matrix function, investigate some of its features, establish a couple of its relations with the available ones in the literature, graphically collate it with a matter of them(see Section 3);
\item we research for a representation of solutions to the nabla Riemann Liouville fractional delayed difference system with noncommutative coefficient matrices in the homogeneous and nonhomogeneous forms(see Section 4);
\item we share some valuable results(see Section 5);
\item we express a few of open problems(see Section 6).
\end{itemize}

\section{ Preliminaries }
In this section, we present the available tools in the literature.

$\mathbb{N}^a=\left\{ \dots, a-2, a-1, a \right\}$, $\mathbb{N}_a=\left\{ a, a+1, a+2, \dots\right\}$, $\mathbb{N}_a^b=\left\{ a, a+1, a+2, \dots, b\right\}$ where $a,b \in \mathbb{R}$ with $b-a \in \mathbb{N}_1 $.

\begin{definition}\cite[Definition 3.4]{p1}
  The generalized rising function is defined by
  \begin{equation*}
    k^{\bar{r}}=\frac{\Gamma(k+r)}{\Gamma(k)},
  \end{equation*}
whenever the right-hand side of this equation is sensible for those values of $k$ and $r$.
\end{definition}

\begin{definition}\cite[Definition 3.56]{p1}
Assume that $\alpha\notin \mathbb{N}^{-1}$. The $\alpha$-th order (nabla) fractional Taylor monomial $H_\alpha(k,a)$ is defined by
  \begin{equation*}
    H_\alpha(k,a)=\frac{(k-a)^{\bar{\alpha}}}{\Gamma(\alpha+1)},
  \end{equation*}
where the right-hand side of the above equation is sensible.
\end{definition}

In the following definitions, we present the (nabla) Leibniz Formula, the (nabla) fractional sum with respect to the (nabla) fractional Taylor monomial and the (nabla) Riemann-Liouville fractional difference, respectively.

\begin{theorem}\cite[Theorem 3.41]{p1}\label{lemmaleibniz}
Assume  $z:\mathbb{N}_{a}\times\mathbb{N}_{a+1}\rightarrow\mathbb{R}$. Then
\begin{equation*}
\nabla\left(\int_{a}^{k} z(k,s)\nabla s \right) =\int_{a}^{k} \nabla_k z(k,s)\nabla s+ z(\rho(k),k), \ \  k\in \mathbb{N}_{a+1}.
\end{equation*}
\end{theorem}
\begin{definition}\cite[Definition 3.58]{p1}
  Let $z:\mathbb{N}_{a+1}\rightarrow\mathbb{R}$ and $\alpha>0$. Then the (nabla) fractional sum is given by
  \begin{equation*}
    \nabla_a^{-\alpha}z(k)=\int_{a}^{k} H_{\alpha-1}(k,\rho(s))z(s)\nabla s =\sum_{s=a+1}^{k} H_{\alpha-1}(k,\rho(s))z(s), \ \ k\in \mathbb{N}_a,
  \end{equation*}
  where $\rho(s)=s-1$.
\end{definition}

\begin{definition}\cite[Definition 3.61, Theorem 3.62]{p1}
Let $z:\mathbb{N}_{a+1}\rightarrow\mathbb{R}$. Then the (nabla) Riemann-Liouville fractional difference of order $0<\alpha<1$ is given by
 \begin{equation*}
    \nabla_a^{\alpha}z(k)=\nabla \nabla_a^{-(1-\alpha)}z(k)= \int_{a}^{k} H_{-\alpha-1}(k,\rho(s))z(s)\nabla s =\sum_{s=a+1}^{k} H_{-\alpha-1}(k,\rho(s))z(s),
  \end{equation*}
for $k\in \mathbb{N}_a$.
\end{definition}

The composition of two different (nabla) fractional sums is given in the below formula of the statement of the coming theorem.
\begin{theorem}\cite[Theorem 3.107]{p1}\label{composition1}
Let $z:\mathbb{N}_{a+1}\rightarrow\mathbb{R}$ and $\alpha, \beta > 0$. Then
\begin{equation*}
   \nabla_a^{-\alpha}\nabla_a^{-\beta}z(k)=\nabla_a^{-\alpha-\beta}z(k), \ \ k\in \mathbb{N}_a.
\end{equation*}
\end{theorem}

In the reference\cite{p1}, the proof of the following lemma is done via the nabla Laplace integral transform. We will prove the lemma in a clear and simple way, compared to the available studies in the literature.
\begin{lemma}\label{lemma2}
  Let $\alpha>0$ and $\beta\in\mathbb{R}$ such that $\beta-1$ and $\alpha+\beta-1$ are nonnegative integers. Then
  \begin{equation*}
    \nabla_a^{-\alpha}H_{\beta-1}(k,a)=H_{\alpha+\beta-1}(k,a), \ \ k \in \mathbb{N}_a,
  \end{equation*}
and
  \begin{equation*}
    \nabla_a^{\alpha}H_{\beta-1}(k,a)=H_{\alpha-\beta-1}(k,a), \ \ k \in \mathbb{N}_a.
  \end{equation*}
\end{lemma}
\begin{proof}
We firstly find an equality of $H_{\beta-1}(k,a)$ using (generalized) binomial coefficients.
\begin{equation*}
 H_{\beta-1}(k,a)=\frac{(k-a)^{\overline{\beta-1}}}{\Gamma(\beta)}=\frac{\Gamma(k-a+\beta-1)}{\Gamma(k-a)\Gamma(\beta)}=\binom{k-a-1+\beta-1}{k-a-1}.
\end{equation*}
Using the above equality and the definition of the (nabla) fractional sum, we acquire
\begin{align*}
\nabla_a^{-\alpha}H_{\beta-1}(k,a)   & = \int_{a}^{k} H_{\alpha-1}(k,\rho(s))H_{\beta-1}(s,a)\nabla s  \\
   & =\sum_{s=a+1}^{k} H_{\alpha-1}(k,\rho(s))H_{\beta-1}(s,a) \\
      & =\sum_{s=a+1}^{k} \binom{k-s+\alpha-1}{k-s}\binom{s-a-1+\beta-1}{s-a-1} \\
            & =\sum_{s=0}^{k-a-1} \binom{k-a-1-s+\alpha-1}{k-a-1-s}\binom{s+\beta-1}{s},
\end{align*}
Applying Chu-Vandermonde identity to the above expression, we get
\begin{align*}
\nabla_a^{-\alpha}H_{\beta-1}(k,a)   & = \int_{a}^{k} H_{\alpha-1}(k,\rho(s))H_{\beta-1}(s,a)\nabla s  \\
   & =\binom{k-a-1+(\alpha+\beta-1)-1}{k-a-1}=H_{\alpha+\beta-1}(k,a).
\end{align*}
The second one can be proved in the similar manner.
\end{proof}

The composition of the (nabla) fractional sum and  the (nabla) Riemann-Liouville fractional difference is given in the next lemma. A special case of the lemma is examined in reference\cite{p1} and it is proved via integration by parts. So we prove it in a different way.
\begin{lemma}\label{composition2}
  For $0<\beta<1$ and $t \in \mathbb{N}_{a+1}$, we have
  \begin{equation*}
    \nabla_a^{-\beta}  \nabla_{\rho(a)}^{\beta}z(k)=z(k)-H_{\beta-1}(k,\rho(a))z(a).
  \end{equation*}
\end{lemma}
\begin{proof}
By using definitions of $ \nabla_a^{-\beta}$ and $ \nabla_a^{\beta}$, we acquire
\begin{align*}
  \nabla_a^{-\beta}  \nabla_{\rho(a)}^{\beta}z(k)  &= \nabla_a^{-\beta}  \left( \sum_{s=a}^{k} H_{-\beta-1}(k,\rho(s))z(s) \right)  \\
   & = \sum_{s=a+1}^{k} H_{\beta-1}(k,\rho(s)) \sum_{i=a}^{s} H_{-\beta-1}(s,\rho(i))z(i)\\
   & = \sum_{s=a}^{k} H_{\beta-1}(k,\rho(s)) \sum_{i=a}^{s} H_{-\beta-1}(s,\rho(i))z(i)\\
   &\ \ - H_{\beta-1}(k,\rho(a)) \sum_{i=a}^{a} H_{-\beta-1}(a,\rho(i))z(i)\\
    & = \sum_{i=a}^{k}z(i) \sum_{s=i}^{k}H_{\beta-1}(k,\rho(s)) H_{-\beta-1}(s,\rho(i))\\
   &\ \ - H_{\beta-1}(k,\rho(a)) H_{-\beta-1}(a,\rho(a))z(a)\\
& = \sum_{i=a}^{k}z(i) \int_{\rho(i)}^{k}H_{\beta-1}(k,\rho(s)) H_{-\beta-1}(s,\rho(i))\nabla s\\
&\ \ - H_{\beta-1}(k,\rho(a)) z(a)\\
  & = \sum_{i=a}^{k} H_{-1}(k,\rho(i))z(i) - H_{\beta-1}(k,\rho(a)) z(a) \\
    & =z(k) - H_{\beta-1}(k,\rho(a)) z(a),
\end{align*}
where $H_{-1}(k,\rho(i))=0$ for $a \leq i \leq k-1$ and $H_{-1}(k,\rho(k))=1$.
\end{proof}

\section{Discrete DPML function}
In this section, the discrete DPML function is newly presented to obtain the exact analytical solution formulas of the homogeneous and nonhomogeneous delayed Riemann-Liouville fractional difference system. Some of its properties to be implemented in the forthcoming proofs are debated.

For noncommutative coefficient constant matrices, the perturbational matrix equation embed in the discrete DPML matrix function  is used in the following definition.
\begin{definition}\label{d-ddp}
The discrete DPML matrix function  ${\mathbb{D}_{\alpha,\beta,r}^{M,N}}$ generated by $M,N$ is defined as follows
{\footnotesize\begin{equation}
{\mathbb{D}_{\alpha,\beta,r}^{M,N}}\left(  k\right)  :=\left\{
\begin{tabular}
[c]{ll}%
$\Theta,$ & $k \in\mathbb{N}^{-r-1},$\\
$I,$ & $k=-r,$\\
${\displaystyle\sum\limits_{i=0}^{\infty}} A^i \dfrac{\left(  k+r\right)  ^{\overline{i\alpha+\beta-1}}%
}{\Gamma\left(  i\alpha+\beta\right)  },$ & $k\in\mathbb{N}_{1-r}^0,$\\
$%
{\displaystyle\sum\limits_{i=0}^{\infty}}
{\displaystyle\sum\limits_{j=0}^{p}}
Q\left( i+1, j\right)  \dfrac{\left(  k-(j-1)r\right)  ^{\overline{i\alpha+\beta-1}}%
}{\Gamma\left(  i\alpha+\beta\right)  },$ & $k\in\mathbb{N}_{(p-1)r+1}^{pr}$,%
\end{tabular}
\ \ \ \ \ \ \ \right.  \label{ddp}%
\end{equation}}
where $\Theta$ and $I$ is the zero and identity matrices,respectively and the matrix equation $Q\left( i, j\right)$ is of the following recursive form
\begin{equation}\label{mainmatrixpart1}
  Q\left( i+1, j\right)=MQ\left( i, j\right)+NQ\left( i, j-1\right), \
\end{equation}
and
\begin{equation}\label{mainmatrixpart2}
   Q(0,j)=Q(i,-1)=\Theta, \ Q(1,0)=I,
\end{equation}
for $i,j\in \mathbb{N}_0.$
\end{definition}
\begin{remark}
  It must be indicated that $Q\left( i, j\right)$ was exploited in the refences\cite{i8} and \cite{d1} to define delayed perturbation of discrete matrix exponential and DPML matrix function, respectively. By employing the above recursive equation, one can easily reach to the explicit form in the below table:
  \footnotesize{\[%
\begin{tabular}
[c]{|c|c|c|c|c|c|c|}\hline
& $j=0$ & $j=1$ & $j=2$ & $j=3$ & $\cdots$ & $j=p$\\\hline
$Q\left( 0, j\right)  $ & $I$ & $\Theta$ & $\Theta$ & $\Theta$ & $\cdots$ & $\Theta%
$\\\hline
$Q\left( 1, j\right)  $ & $M$ & $N$ & $\Theta$ & $\Theta$ & $\cdots$ & \\\hline
$Q\left( 2, j\right)  $ & $M^2$ & $MN+NM$ & $N^{2}$ & $\Theta$ & $\cdots$ & $\Theta$\\\hline
$Q\left(3, j\right)  $ & $M^3$ & $M(MN+NM)+NM^2$ & $MN^2+N(MN+NM)$ & $N^{3}$ &
$\cdots$ & \\\hline
$\cdots$ & $\cdots$ & $\cdots$ & $\cdots$ &  & $\cdots$ & $\Theta$\\\hline
$Q\left( p, j\right)  $ & $M^p$ & $\Theta$ & $\Theta$ & $\Theta$ &
$\cdots$ & $N^{p}$\\\hline
\end{tabular}
\
\]}
\end{remark}
We can reexpress the discrete DPML matrix function  ${\mathbb{D}_{\alpha,\beta,r}^{M,N}}$ in terms of the (nabla) fractional Taylor monomial as follows $H_\alpha(k,a)$.
{\footnotesize\begin{equation}
{\mathbb{D}_{\alpha,\beta,r}^{M,N}}\left(  k\right)  :=\left\{
\begin{tabular}
[c]{ll}%
$\Theta,$ & $k \in\mathbb{N}^{-r-1},$\\
$I,$ & $k=-r,$\\
${\displaystyle\sum\limits_{i=0}^{\infty}} M^i H_{i\alpha+\beta-1}\left(k,-r\right),$ & $k\in\mathbb{N}_{1-r}^0,$\\
$%
{\displaystyle\sum\limits_{i=0}^{\infty}}
{\displaystyle\sum\limits_{j=0}^{p}}
Q\left( i+1, j\right) H_{i\alpha+\beta-1}\left(k,(j-1)r\right)  ,$ & $k\in\mathbb{N}_{(p-1)r+1}^{pr}$.%
\end{tabular}
\ \ \ \ \ \ \ \right.  \label{ddp2}%
\end{equation}}

For commutative constant coefficient matrices, the following lemma is debated.

\begin{lemma}
Under the commutativity of the constant coefficient matrices $M$ ad $N$, we have
\begin{equation*}
  Q\left( i, j\right)=\binom{i}{j}M^{i-j}N^j, \ \ i,j \in \mathbb{N}_0.
\end{equation*}
\end{lemma}
\begin{proof}
We apply the mathematical induction on $j\in\mathbb{N}_0$ to the recursive equation to prove this lemma.
Let us see that it is true for $j=0$.
\begin{align*}
  Q\left( i, 0\right)&=MQ\left( i-1, 0\right)+NQ\left( i-1, -1\right)  \\
   &= MQ\left( i-1, 0\right)=M^i=\binom{i}{0}M^{i-0}N^0.
\end{align*}
Assume that it is true for $j=n$, that is,
\begin{equation*}
  Q\left( i, n\right)=\binom{i}{n}M^{i-n}N^n.
\end{equation*}
Let us check its validity for $j=n+1$.
\begin{align*}
  Q\left( i, n+1\right)&=MQ\left( i-1, n+1\right)+NQ\left( i-1, n\right)  \\
   &= M\binom{i-1}{n+1}M^{i-1-n-1}N^{n+1}+N\binom{i-1}{n}M^{i-1-n}N^{n}\\
   &=\left[\binom{i-1}{n+1}+\binom{i-1}{n}\right]M^{i-1-n}N^{n+1}\\
   &=\binom{i}{n+1}M^{i-(n+1)}N^{n+1}.
\end{align*}
\end{proof}

In the following remark, it is seen that the discrete DPML matrix function we define reduces to the available famous functions in the literature depending on the special choices of the parameters.
\begin{remark}
Let ${\mathbb{D}_{\alpha,\beta,r}^{M,N}}$  be defined by (\ref{ddp}). Then the following assertions hold true.
\begin{enumerate}
  \item  if $\alpha=\beta=1$, $r=h+1$, then ${\mathbb{D}_{\alpha,\beta,h}^{M,N}}\left(k\right)=X_{h}^{M,N}\left(k\right)$,
  \item if $M=\Theta$, $r=h+1$, and $\alpha=\beta=1$, then ${\mathbb{D}_{\alpha,\beta,h}^{M,N}}\left(k\right)=e_{h}^{Nk}$,
  \item if $M=\Theta$  and $\alpha=\beta$, then ${\mathbb{D}_{\alpha,\beta,r}^{M,N}}\left(k\right)=\mathbb{F}_r^{Nk^{\overline{\alpha}}}$,
  \item if $N=\Theta$, then ${\mathbb{D}_{\alpha,\beta,r}^{M,N}}\left(k\right)=\mathbb{E}_{M,\alpha,\beta-1}(k,-r)$,
  \item if $\alpha=\beta=1$, $r=h+1$, and $MN=NM$, then
  ${\mathbb{D}_{\alpha,\beta,h}^{M,N}}\left(k\right)=e^{Mk}e_{h}^{N_1(k-h)}$,
  where $N_1=e^{-Mh}N$.
\end{enumerate}
\end{remark}
This remark can be easily confirmed by keeping aforesaid definitions in mind.

\begin{remark}
  Comparison of the discrete delay perturbation of (nabla) M-L type matrix function $ {\mathbb{D}_{\alpha,\beta,r}^{M,N}}\left(k\right)$, the (nabla) M-L matrix function $\mathbb{E}_{M,\alpha,\beta}(k,r)$, and discrete delayed M-L-type matrix function $ \mathbb{F}_r^{Nk^{\overline{\alpha}}}$ is presented in Figure \ref{comparison_of_functions}.
\end{remark}
\begin{figure}
  \centering
  \includegraphics[height=7cm, width=12cm]{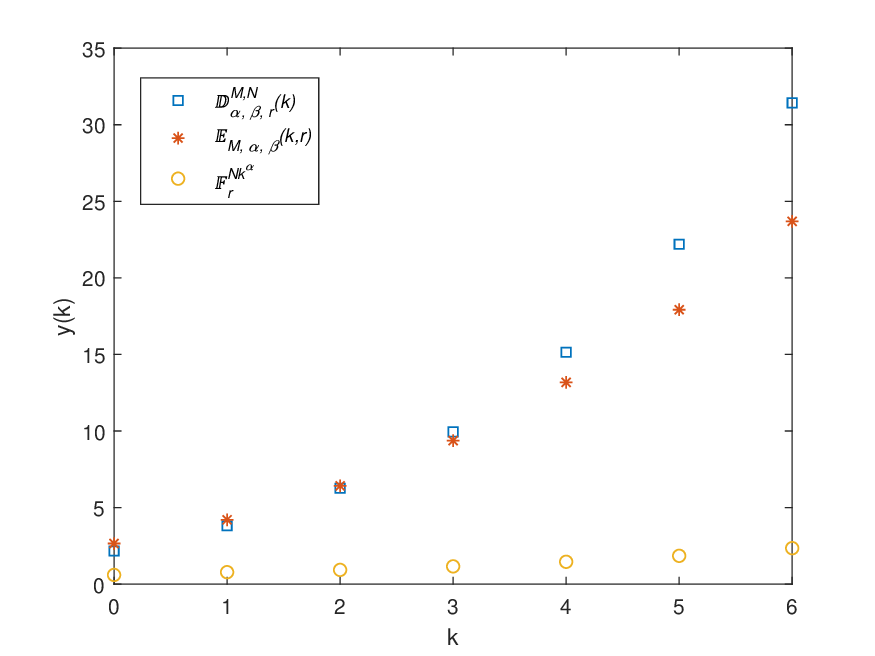}
  \caption{Comparison of functions $ {\mathbb{D}_{\alpha,\beta,r}^{M,N}}\left(k\right)$, $\mathbb{E}_{M,\alpha,\beta}(k,r)$, and $ \mathbb{F}_r^{Nk^{\overline{\alpha}}}$ for $\alpha=0.9$, $\beta=0.6$, $M=5$, $N=3$, $r=2$.}\label{comparison_of_functions}
\end{figure}

\section{The explicit solution of RL fractional retarded difference system }
In this section, our aim is to investigate an explicit solution to the linear Riemann Liouville fractional retarded difference system by dividing into three subsections.

We share a couple of  main theorems to achieve our objective. When it comes to most of their proofs related to natural numbers, we use the method of mathematical induction for them, which is the most useful and powerful proof technique on the natural numbers.

\subsection{Homogeneous case:}
In this subsection, we firstly consider the following homogeneous Riemann Liouville fractional delayed difference system
\begin{equation}
\left\{
\begin{array}
[c]{l}%
\nabla_{-r}^\alpha z \left( k\right) =Mz\left(  k\right)  +Nz\left(
k-r\right),\ \ k\in\mathbb{N}_1  ,\\
\ \ \ \ z\left(  k\right)  =\mathbb{E}_{M,\alpha,\alpha-1}\left(k,-r \right)  ,\ \ k\in\mathbb{N}_{1-r}^0, \ \ r>0,\\
\end{array}
\right.  \label{sorkokhomogeneous}%
\end{equation}
where $\nabla_{-r}^\alpha$ is the Riemann-Liouville fractional difference of order $0<\alpha<1$, $z:\mathbb{N}_1\rightarrow\mathbb{R}^n$,  $r \in \mathbb{N}_2$ is a retardation, $M, N \in \mathbb{R}^{n\times n}$ are constant coefficient matrices.

\begin{theorem}\label{maintheorem1}
  The discrete DPML matrix function  ${\mathbb{D}_{\alpha,\beta,r}^{M,N}}$ satisfies homogeneous system (\ref{sorkokhomogeneous}) in the case $\alpha=\beta$.
\end{theorem}
\begin{proof}
Firstly we show that it verifies the initial condition $z\left(  k\right)  =\mathbb{E}_{A,\alpha,\alpha-1}\left(k,-r \right)$, for $k\in\mathbb{N}_{1-r}^0$. From Definition \ref{ddp2}, one can easily see that
\begin{equation*}
 {\mathbb{D}_{\alpha,\alpha,r}^{M,N}}(k)= {\displaystyle\sum\limits_{i=0}^{\infty}} M^i H_{i\alpha+\alpha-1}\left(k,-r\right)=\mathbb{E}_{M,\alpha,\alpha-1}\left(k,-r \right).
\end{equation*}
Now, for $ k\in\mathbb{N}_1$ we show that
\begin{equation}\label{mus1}
  \nabla_{-r}^\alpha {\mathbb{D}_{\alpha,\alpha,r}^{M,N}}(k) =M{\mathbb{D}_{\alpha,\alpha,r}^{M,N}}(k)  +N{\mathbb{D}_{\alpha,\alpha,r}^{M,N}}(k-r).
\end{equation}
For $k\in \mathbb{N}_1$, there exists $p\in \mathbb{N}_1$ such that $k\in\mathbb{N}_{(p-1)r+1}^{pr}$. We apply the mathematical induction on $p\in \mathbb{N}_1$ to demonstrate its trueness. For $p=1$, we have
\begin{equation}\label{mainthm1}
{\mathbb{D}_{\alpha,\alpha,r}^{M,N}}(k)=  {\displaystyle\sum\limits_{i=0}^{\infty}} {\displaystyle\sum\limits_{j=0}^{1}} Q\left( i+1, j\right) H_{i\alpha+\alpha-1}\left(k,(j-1)r\right).
\end{equation}
Taking the Riemann Liouville fractional difference $\nabla_{-r}^\alpha$ of (\ref{mainthm1}), we acquire
\begin{align*}
\nabla_{-r}^\alpha  {\mathbb{D}_{\alpha,\alpha,r}^{M,N}}(k) & =   {\displaystyle\sum\limits_{i=0}^{\infty}} {\displaystyle\sum\limits_{j=0}^{1}} Q\left( i+1, j\right)\nabla_{-r}^\alpha H_{i\alpha+\alpha-1}\left(k,(j-1)r\right).
\end{align*}
Taking the subintervals and Lemma \ref{lemma2}  into consideration, we have
\begin{align*}
\nabla_{-r}^\alpha  {\mathbb{D}_{\alpha,\alpha,r}^{M,N}}(k) & =   {\displaystyle\sum\limits_{i=0}^{\infty}} {\displaystyle\sum\limits_{j=0}^{1}} Q\left( i+1, j\right)\nabla_{(j-1)r}^\alpha H_{i\alpha+\alpha-1}\left(k,(j-1)r\right)\\
&=   {\displaystyle\sum\limits_{i=0}^{\infty}} {\displaystyle\sum\limits_{j=0}^{1}} Q\left( i+1, j\right) H_{i\alpha-1}\left(k,(j-1)r\right).
\end{align*}
Using (\ref{mainmatrixpart1}) and (\ref{mainmatrixpart2}), we acquire
\begin{align*}
\nabla_{-r}^\alpha  {\mathbb{D}_{\alpha,\alpha,r}^{M,N}}(k) & =   M{\displaystyle\sum\limits_{i=1}^{\infty}} {\displaystyle\sum\limits_{j=0}^{1}} Q\left( i, j\right) H_{i\alpha-1}\left(k,(j-1)r\right)\\
&+   N{\displaystyle\sum\limits_{i=1}^{\infty}} {\displaystyle\sum\limits_{j=1}^{1}} Q\left( i, j-1\right) H_{i\alpha-1}\left(k,(j-1)r\right)\\
& =   M{\displaystyle\sum\limits_{i=0}^{\infty}} {\displaystyle\sum\limits_{j=0}^{1}} Q\left( i+1, j\right) H_{i\alpha+\alpha-1}\left(k,(j-1)r\right)\\
&+   N{\displaystyle\sum\limits_{i=0}^{\infty}} {\displaystyle\sum\limits_{j=0}^{0}} Q\left( i+1, j\right) H_{i\alpha+\alpha-1}\left(k,jr\right)\\
&=M{\mathbb{D}_{\alpha,\alpha,r}^{M,N}}(k)  +N{\mathbb{D}_{\alpha,\alpha,r}^{M,N}}(k-r).
\end{align*}
Now, let us assume its validity for $p=n$, that is
\begin{equation*}
{\mathbb{D}_{\alpha,\alpha,r}^{M,N}}(k)=  {\displaystyle\sum\limits_{i=0}^{\infty}} {\displaystyle\sum\limits_{j=0}^{n}} Q\left( i+1, j\right) H_{i\alpha+\alpha-1}\left(k,(j-1)r\right)
\end{equation*}
satisfies (\ref{mus1}). For $p=n+1$, consider via the similar calculations as in the first case
\begin{align*}
\nabla_{-r}^\alpha  {\mathbb{D}_{\alpha,\alpha,r}^{M,N}}(k) & =   {\displaystyle\sum\limits_{i=0}^{\infty}} {\displaystyle\sum\limits_{j=0}^{n+1}} Q\left( i+1, j\right)\nabla_{(j-1)r}^\alpha H_{i\alpha+\alpha-1}\left(k,(j-1)r\right)\\
&=   {\displaystyle\sum\limits_{i=0}^{\infty}} {\displaystyle\sum\limits_{j=0}^{n+1}} Q\left( i+1, j\right) H_{i\alpha-1}\left(k,(j-1)r\right)\\
& =   M{\displaystyle\sum\limits_{i=1}^{\infty}} {\displaystyle\sum\limits_{j=0}^{n+1}} Q\left( i, j\right) H_{i\alpha-1}\left(k,(j-1)r\right)\\
&+   N{\displaystyle\sum\limits_{i=1}^{\infty}} {\displaystyle\sum\limits_{j=1}^{n+1}} Q\left( i, j-1\right) H_{i\alpha-1}\left(k,(j-1)r\right)\\
&=M{\mathbb{D}_{\alpha,\alpha,r}^{M,N}}(k)  +N{\mathbb{D}_{\alpha,\alpha,r}^{M,N}}(k-r),
\end{align*}
which is the foreseen result.
\end{proof}

Secondly, we examine the below homogeneous Riemann Liouville fractional delayed difference system
\begin{equation}
\left\{
\begin{array}
[c]{l}%
\nabla_{-r}^\alpha z \left( k\right) =Mz\left(  k\right)  +Nz\left(
k-r\right),\ \ k\in\mathbb{N}_1  ,\\
\ \ \ \ z\left(  k\right)  =\phi\left(  k\right)  ,\ \ k\in\mathbb{N}_{1-r}^0, \ \ r>0,\\
\end{array}
\right.  \label{sorkokhomogeneous2}%
\end{equation}
$\phi:\mathbb{N}_{1-r}^0\rightarrow\mathbb{R}^n$ is an initial conditional function, the rest of details are same as (\ref{sorkokhomogeneous2}).

\begin{theorem}\label{maintheoremhomogeneous2}
  The following function
  \begin{equation*}\label{solution2}
    z(k)={\mathbb{D}_{\alpha,\alpha,r}^{M,N}}(k)\phi(1-r)
    +\int_{1-r}^{0}{\mathbb{D}_{\alpha,\alpha,r}^{M,N}}\left(k-r-\rho(s)\right)\left(\nabla_{-r}^\alpha \phi(s)-M\phi(s)\right)\nabla s,
  \end{equation*}
  satisfies  homogeneous system (\ref{sorkokhomogeneous2}).
\end{theorem}
\begin{proof}
  In the light of Definition \ref{ddp2}, one easily acquire  that,  	with an eye to ${\mathbb{D}_{\alpha,\alpha,r}^{M,N}}\left(k-r-\rho(s)\right)=0$, for $k\in\mathbb{N}_{1-r}^0$ ,
  \begin{align*}
    z(k)  & =  {\mathbb{D}_{\alpha,\alpha,r}^{M,N}}(k)\phi(1-r)+\int_{1-r}^{k}{\mathbb{D}_{\alpha,\alpha,r}^{M,N}}\left(k-r-\rho(s)\right)\left(\nabla_{-r}^\alpha \phi(s)-M\phi(s)\right)\nabla s\\
     & ={\displaystyle\sum\limits_{i=0}^{\infty}} M^i  H_{i\alpha+\alpha-1}\left(k,-r\right)\phi(1-r) \\
     &+{\displaystyle\sum\limits_{i=0}^{\infty}} M^i\int_{1-r}^{k}H_{i\alpha+\alpha-1}\left(k,\rho(s)\right)\left(\nabla_{-r}^\alpha \phi(s)-M\phi(s)\right)\nabla s\\
     &={\displaystyle\sum\limits_{i=0}^{\infty}} M^i H_{i\alpha+\alpha-1}\left(k,-r\right)\phi(1-r) \\
     &+{\displaystyle\sum\limits_{i=0}^{\infty}} M^i \left(\nabla_{1-r}^{-(i\alpha+\alpha)}\nabla_{-r}^\alpha \phi(k)-M\nabla_{1-r}^{-(i\alpha+\alpha)}\phi(k)\right).
  \end{align*}
It is time to apply Theorem \ref{composition1} and Lemma \ref{composition2} to the above equation, so we get
 \begin{align*}
    z(k)
     &={\displaystyle\sum\limits_{i=0}^{\infty}} M^i H_{i\alpha+\alpha-1}\left(k,-r\right)\phi(1-r)\\
     &={\displaystyle\sum\limits_{i=0}^{\infty}} M^i H_{i\alpha+\alpha-1}\left(k,-r\right)\phi(1-r)-M{\displaystyle\sum\limits_{i=0}^{\infty}}M^i \nabla_{1-r}^{-(i\alpha+\alpha)}\phi(k) \\
     &+{\displaystyle\sum\limits_{i=0}^{\infty}}M^i  \left(\nabla_{1-r}^{-i\alpha} \left( \phi(k)-H_{i\alpha+\alpha-1}\left(k,-r\right)\phi(1-r) \right)\right)\\
     & = {\displaystyle\sum\limits_{i=1}^{\infty}} M^i \nabla_{1-r}^{-i\alpha}  \phi(k) + \phi(k) -M{\displaystyle\sum\limits_{i=0}^{\infty}}M^i \nabla_{1-r}^{-(i\alpha+\alpha)}\phi(k)
     =\phi(k),
  \end{align*}
 which means that $z(k)$ satisfies the initial condition  for $k\in\mathbb{N}_{1-r}^0$.  In order to demonstrate that $z(k)$ fulfills homogeneous system $(\ref{composition2})$ for $k\in\mathbb{N}_1$, we exploit the Riemann Liouville fractional difference, the well-known Fubini Theorem, and Theorem \ref{maintheorem1} to acquire:
 \begin{align*}
   \nabla_{-r}^\alpha z \left( k\right)  &= \nabla_{-r}^\alpha {\mathbb{D}_{\alpha,\alpha,r}^{M,N}}(k)\phi(1-r)\\
    &+\int_{1-r}^{0}\nabla_{-r}^\alpha{\mathbb{D}_{\alpha,\alpha,r}^{M,N}}\left(k-r-\rho(s)\right)\left(\nabla_{-r}^\alpha \phi(s)-M\phi(s)\right)\nabla s   \\
    &= M{\mathbb{D}_{\alpha,\alpha,r}^{M,N}}(k)\phi(1-r)+N{\mathbb{D}_{\alpha,\alpha,r}^{M,N}}(k-r)\phi(1-r)\\
    &+M\int_{1-r}^{0}{\mathbb{D}_{\alpha,\alpha,r}^{M,N}}\left(k-r-\rho(s)\right)\left(\nabla_{-r}^\alpha \phi(s)-M\phi(s)\right)\nabla s   \\
    &+N\int_{1-r}^{0}{\mathbb{D}_{\alpha,\alpha,r}^{M,N}}\left(k-2r-\rho(s)\right)\left(\nabla_{-r}^\alpha \phi(s)-M\phi(s)\right)\nabla s \\
    &=Mz\left(  k\right)  +Nz\left(k-r\right).
 \end{align*}
 which completes the proof.
 \end{proof}

\begin{remark}
Since the solution of the homogeneous part is known from Theorem \ref{maintheorem1}, the variation of constants' technique also can be used to prove this theorem.
\end{remark}

\subsection{Nonhomogeneous case:}
In this subsection, we  consider the following nonhomogeneous Riemann Liouville fractional delayed difference system
\begin{equation}
\left\{
\begin{array}
[c]{l}%
\nabla_{-r}^\alpha z \left( k\right) =Mz\left(  k\right)  +Nz\left(
k-r\right)  +\daleth\left(  k  \right)   ,\ \ k\in\mathbb{N}_1  ,\\
\ \ \ \ z\left(  k\right)  =0  ,\ \ k\in\mathbb{N}_{1-r}^0, \ \ r>0,\\
\end{array}
\right.  \label{sorkoknonhomogeneous}%
\end{equation}
where $\nabla_{-r}^\alpha$ is the Riemann-Liouville fractional difference of order $0<\alpha<1$, $z:\mathbb{N}_1\rightarrow\mathbb{R}^n$, $\daleth:\mathbb{N}_1\rightarrow\mathbb{R}^n$ is a continuous function, $r \in \mathbb{N}_2$ is a retardation, $M, N \in \mathbb{R}^{n\times n}$ are constant coefficient matrices.

Before giving one of main theorems, we prove some useful equalities in the following lemma.

\begin{lemma}\label{lemmanonhomogeneous}
  The below equalities hold good.
  \begin{itemize}
    \item $\displaystyle\int_{-r}^{0} H_{-\alpha}\left(k,\rho(s)\right) \int_{0}^{s} {\mathbb{D}_{\alpha,\alpha,r}^{M,N}}\left(s-r-\rho(r)\right)\daleth(r)\nabla \rho(r) \nabla s=0$,
    \item $\nabla_{\rho(k)}^{\alpha-1}\daleth(k)=\daleth(k)$.
  \end{itemize}
\end{lemma}
\begin{proof}
  For the first item, take $s=0$, the whole term should be zero from the inner integral whose borders are equal to zero. If $0 \leq \rho(r)\leq s$ and $ -r\leq s <0$, then $s-r-\rho(r) \leq -r-1$, and so ${\mathbb{D}_{\alpha,\alpha,r}^{M,N}}\left(s-r-\rho(r)\right)=\Theta$ because of Definition \ref{ddp2}. The whole term again should be zero. For the second item, we need  the simple calculation as follows:
  \begin{align*}
    \nabla_{\rho(k)}^{\alpha-1}\daleth(k) & =\sum_{s=\rho(k)+1}^{k} H_{\alpha-1}(k,\rho(s))\daleth(s)
      =\sum_{s=k}^{k} H_{\alpha-1}(k,\rho(s))\daleth(s)  \\
     & = H_{\alpha-1}(k,\rho(k))\daleth(k)
      = \daleth(k),
  \end{align*}
where $H_{\alpha-1}(k,\rho(k))=1$.
\end{proof}

\begin{theorem}\label{maintheoremnonhomogeneous}
The following integral term
  \begin{equation*}
    z(k)=\int_{0}^{k} {\mathbb{D}_{\alpha,\alpha,r}^{M,N}}\left(k-r-\rho(s)\right)\daleth(s)\nabla s,
  \end{equation*}
fulfills nonhomogeneous system $(\ref{sorkoknonhomogeneous})$.
\end{theorem}
\begin{proof}
We will firstly  show, for $k \in \mathbb{N}_{1-r}^0$, that
\begin{equation*}
    z(k)=\int_{0}^{k} {\mathbb{D}_{\alpha,\alpha,r}^{M,N}}\left(k-r-\rho(s)\right)\daleth(s)\nabla s=0.
  \end{equation*}
Due to the property of integral, $ z(k)=\int_{0}^{0} {\mathbb{D}_{\alpha,\alpha,r}^{M,N}}\left(k-r-\rho(s)\right)\daleth(s)\nabla s=0$, for $k=0$.  From \cite[Theorem 8.48]{a1}, one verifies that
\begin{align*}
  z(k) &  =\int_{0}^{k} {\mathbb{D}_{\alpha,\alpha,r}^{M,N}}\left(k-r-\rho(s)\right)\daleth(s)\nabla s \\
   & =-\int_{k}^{0} {\mathbb{D}_{\alpha,\alpha,r}^{M,N}}\left(k-r-\rho(s)\right)\daleth(s)\nabla s \\
    & =-\sum_{s=k+1}^{0} {\mathbb{D}_{\alpha,\alpha,r}^{M,N}}\left(k-r-\rho(s)\right)\daleth(s).
\end{align*}
If $k< s \leq 0$ and $1-r \leq k \leq -1$ together with a property of integer arithmetic, then $k-r-\rho(s) \in \mathbb{N}^{-1-r}$, and so ${\mathbb{D}_{\alpha,\alpha,r}^{M,N}}\left(k-r-\rho(s)\right)=\Theta$. Thus, we have $z(k)=0$ for $\mathbb{N}_{1-r}^{-1}$. To sum up, $z(k)=0$ for $\mathbb{N}_{1-r}^{0}$, that is,  $z(k)$ fulfills the initial condition.

In the rest of this proof, we use the technique of mathematical induction on $p \in \mathbb{N}_1$.  For $k\in \mathbb{N}_1$, there exists $p\in \mathbb{N}_1$ such that $k\in\mathbb{N}_{(p-1)r+1}^{pr}$. Suppose that it is true for $p=1$ such that  $k\in\mathbb{N}_{1}^{r}$.

It is clear, from the initial condition, that $z\left(
k-r\right)=0$ for $k\in\mathbb{N}_{1}^{r}$, and so
\begin{equation*}
  Mz\left(  k\right)  +Nz\left(
k-r\right)  +\daleth\left(  k  \right)=Mz\left(  r\right)   +\daleth\left(  r  \right).
\end{equation*}

On the other hand, we have for $k\in\mathbb{N}_{1}^{r}$
\begin{align*}
  \nabla_{-r}^{\alpha}z(k)& =\nabla_{-r}^{\alpha}\left[\int_{0}^{k} {\mathbb{D}_{\alpha,\alpha,r}^{M,N}}\left(k-r-\rho(s)\right)\daleth(s)\nabla s \right]\\
   & =\nabla \nabla_{-r}^{\alpha-1}\left[\int_{0}^{k} {\mathbb{D}_{\alpha,\alpha,r}^{M,N}}\left(k-r-\rho(s)\right)\daleth(s)\nabla s \right]\\
    & =\nabla \int_{-r}^{k}H_{-\alpha}(k,\rho(s))\left(\int_{0}^{s} {\mathbb{D}_{\alpha,\alpha,r}^{M,N}}\left(s-r-\rho(t)\right)\daleth(t)\nabla t \right) \nabla s\\
    & =\nabla \int_{-r}^{k}H_{-\alpha}(k,\rho(s))\left(\int_{0}^{s} {\mathbb{D}_{\alpha,\alpha,r}^{M,N}}\left(s-r-\rho(t)\right)\daleth(t)\nabla \rho(t) \right) \nabla s.
\end{align*}
It is time to use Lemma \ref{lemmanonhomogeneous} and Fubini Theorem, so we get
\begin{align*}
  \nabla_{-r}^{\alpha}z(k)   & =\nabla \int_{0}^{k} \int_{\rho(t)}^{k} H_{-\alpha}(k,\rho(s)) {\mathbb{D}_{\alpha,\alpha,r}^{M,N}}\left(s-r-\rho(t)\right)\daleth(t)\nabla s  \nabla t\\
  &= \nabla \int_{0}^{k}  \nabla_{\rho(t)}^{\alpha-1} {\mathbb{D}_{\alpha,\alpha,r}^{M,N}}\left(s-r-\rho(t)\right)\daleth(t)  \nabla t.
\end{align*}
By implementing Lemma \ref{lemmaleibniz}, Definition \ref{ddp2}, and Lemma \ref{lemmanonhomogeneous}, respectively to the above equality
\begin{align*}
  \nabla_{-r}^{\alpha}z(k)    &=  \int_{0}^{k}  \nabla_{\rho(t)}^{\alpha} {\mathbb{D}_{\alpha,\alpha,r}^{M,N}}\left(s-r-\rho(t)\right)\daleth(t)  \nabla t + \daleth(k)\\
  &= M \int_{0}^{k}  {\mathbb{D}_{\alpha,\alpha,r}^{M,N}}\left(s-r-\rho(t)\right)\daleth(t)  \nabla t + \daleth(k)\\
  &+ N \int_{0}^{k}  {\mathbb{D}_{\alpha,\alpha,r}^{M,N}}\left(s-2r-\rho(t)\right)\daleth(t)  \nabla t \\
  &= Mz(k)+\daleth(k).
\end{align*}
On comparing the obtained results on $k\in\mathbb{N}_{1}^{r}$, we have
\begin{equation*}
   \nabla_{-r}^{\alpha}z(k)=Mz\left(  r\right)   +\daleth\left(  r  \right)=Mz\left(  k\right)  +Nz\left(
k-r\right)  +\daleth\left(  k  \right).
\end{equation*}
Suppose that it is valid for $p=n$ such that $k\in\mathbb{N}_{(n-1)r+1}^{nr}$, that is
\begin{equation*}
 {\mathbb{D}_{\alpha,\alpha,r}^{M,N}}\left(k\right) = {\displaystyle\sum\limits_{i=0}^{\infty}}
{\displaystyle\sum\limits_{j=0}^{n}}
Q\left( i+1, j\right) H_{i\alpha+\alpha-1}\left(k,(j-1)r\right).
\end{equation*}
In the same manner as the first case, we acquire for $k\in\mathbb{N}_{nr+1}^{(n+1)r}$
\begin{align*}
  \nabla_{-r}^{\alpha}z(k)    &=  \int_{0}^{k}  \nabla_{\rho(t)}^{\alpha} {\mathbb{D}_{\alpha,\alpha,r}^{M,N}}\left(s-r-\rho(t)\right)\daleth(t)  \nabla t + \daleth(k)\\
  &= M \int_{0}^{k}  {\mathbb{D}_{\alpha,\alpha,r}^{M,N}}\left(s-r-\rho(t)\right)\daleth(t)  \nabla t + \daleth(k)\\
  &+ N \int_{0}^{k}  {\mathbb{D}_{\alpha,\alpha,r}^{M,N}}\left(s-2r-\rho(t)\right)\daleth(t)  \nabla t \\
 &= M \int_{0}^{k}  {\mathbb{D}_{\alpha,\alpha,r}^{M,N}}\left(s-r-\rho(t)\right)\daleth(t)  \nabla t + \daleth(k)\\
  &+ N \int_{0}^{k-r}  {\mathbb{D}_{\alpha,\alpha,r}^{M,N}}\left(s-2r-\rho(t)\right)\daleth(t)  \nabla t \\
  &=Mz\left(  k\right)  +Nz\left(
k-r\right)  +\daleth\left(  k  \right),
\end{align*}
which is what we want to prove.
\end{proof}
\begin{remark}
Based on the homogeneous solution from Theorem \ref{maintheorem1}, the variation of constants' technique also can be used to prove this theorem.
\end{remark}

Now, we will examine the below nonhomogeneous retarded Riemann Liouville fractional difference system
\begin{equation}
\left\{
\begin{array}
[c]{l}%
\nabla_{-r}^\alpha z \left( k\right) =Mz\left(  k\right)  +Nz\left(
k-r\right)  +\daleth\left(  k  \right)   ,\ \ k\in\mathbb{N}_1  ,\\
\ \ \ \ z\left(  k\right)  =\phi\left(  k\right)  ,\ \ k\in\mathbb{N}_{1-r}^0, \ \ r>0,\\
\end{array}
\right.  \label{sorkoknonhomogeneous2}%
\end{equation}
where $\nabla_{-r}^\alpha$ is the Riemann-Liouville fractional difference of order $0<\alpha<1$, $z:\mathbb{N}_1\rightarrow\mathbb{R}^n$, $\daleth:\mathbb{N}_1\rightarrow\mathbb{R}^n$ is a continuous function, $r \in \mathbb{N}_2$ is a retardation, $M, N \in \mathbb{R}^{n\times n}$ are constant coefficient matrices, $\phi:\mathbb{N}_{1-r}^0\rightarrow\mathbb{R}^n$ is an initial conditional function.

\begin{theorem}\label{mainlasttheorem}
  The whole function
  \begin{align}\label{totalsolution}
    z(k)&={\mathbb{D}_{\alpha,\alpha,r}^{M,N}}(k)\phi(1-r)
    +\int_{1-r}^{0}{\mathbb{D}_{\alpha,\alpha,r}^{M,N}}\left(k-r-\rho(s)\right)\left(\nabla_{-r}^\alpha \phi(s)-M\phi(s)\right)\nabla s \nonumber\\
     & +\int_{0}^{k} {\mathbb{D}_{\alpha,\alpha,r}^{M,N}}\left(k-r-\rho(s)\right)\daleth(s)\nabla s,
  \end{align}
  satisfies nonhomogeneous system $(\ref{sorkoknonhomogeneous2})$
\end{theorem}
\begin{proof}
It is clear that $z(k)=z_1(k)+z_2(k)$ in equation (\ref{totalsolution}), where $z_1(k)$ and  $z_2(k)$ are solutions given Theorem \ref{maintheoremhomogeneous2} and \ref{maintheoremnonhomogeneous} of systems (\ref{sorkokhomogeneous2}) and (\ref{maintheoremnonhomogeneous}), respectively, which form system (\ref{sorkoknonhomogeneous2}). By implementing the principle of superposition technique, $z(k)$ fulfills  system (\ref{sorkoknonhomogeneous2}). Indeed, for $k\in \mathbb{N}_1$ we acquire
\begin{align*}
\nabla_{-r}^\alpha z \left( k\right)   & = \nabla_{-r}^\alpha \left[ z_1(k)+z_2(k)\right]  \\
&=Mz_1\left(  k\right)  +Nz_1\left(k-r\right)  +Mz_2\left(  k\right)  +Nz_2\left(k-r\right)  +\daleth\left(  k  \right)\\
   & =M\left[z_1\left(  k\right)+z_2\left(  k\right) \right]+N\left[ z_2\left(  k-r\right)+z_2\left(  k-r\right)\right]+\daleth\left(  k  \right) \\
   &=Mz\left(  k\right)  +Nz\left(
k-r\right)  +\daleth\left(  k  \right),
\end{align*}
and for $k\in \mathbb{N}_{1-r}^0$,
\begin{align*}
\nabla_{-r}^\alpha z \left( k\right)   & = \nabla_{-r}^\alpha \left[ z_1(k)+z_2(k)\right]  \\
&=\phi(k)+0=\phi(k),
\end{align*}
which is the craved result.
\end{proof}

\begin{remark}
  Theorem \ref{mainlasttheorem}  corresponds to \cite[Theorem 4]{i3}, for $M=\Theta$  and $\alpha=\beta$,
\end{remark}

\section{Special Cases}
In this section we exhibit a couple of special cases.

\begin{example}
Let us reconsider the matrix nabla fractional delayed difference system (\ref{sorkok2}) when $MN=NM$. In this context, Theorem \ref{mainlasttheorem} may be reformulated as follows.
\end{example}

\begin{proposition}\label{proposition1}
The exact analytical solution of nonhomogeneous version of  system (\ref{sorkok2}) has the following explicit form:
{\footnotesize\begin{align*}
z(k)
    &=        {\displaystyle\sum\limits_{i=0}^{\infty}}
{\displaystyle\sum\limits_{j=0}^{p}}
\binom{i+1}{j}M^{i+1-j}N^j H_{(i+1)\alpha-1}\left(k,(j-1)r\right)\phi(1-r)\\
&    +{\displaystyle\sum\limits_{i=0}^{\infty}}
{\displaystyle\sum\limits_{j=0}^{p}}\binom{i+1}{j}M^{i+1-j}N^j \int_{1-r}^{0} H_{(i+1)\alpha-1}\left(k-r-\rho(s),(j-1)r\right)\nabla_{-r}^\alpha \phi(s)\nabla s \\
&    -{\displaystyle\sum\limits_{i=0}^{\infty}}
{\displaystyle\sum\limits_{j=0}^{p}}\binom{i+1}{j}M^{i+2-j}N^j \int_{1-r}^{0} H_{(i+1)\alpha-1}\left(k-r-\rho(s),(j-1)r\right)\phi(s)\nabla s \\
     & +{\displaystyle\sum\limits_{i=0}^{\infty}}
{\displaystyle\sum\limits_{j=0}^{p}}\binom{i+1}{j}M^{i+1-j}N^j \int_{0}^{k} H_{(i+1)\alpha-1}\left(k-r-\rho(s),(j-1)r\right)\daleth(s,z(s))\nabla s.
  \end{align*}}
\end{proposition}

\begin{remark}
Proposition \ref{proposition1} is novel for matrix nabla RL fractional delayed difference system with commutative coefficient matrices.
\end{remark}

\begin{example}
  Let us reconsider delta analogue of the nabla fractional delayed difference system (\ref{sorkok2}) via the relation $\nabla z(k)=\Delta z(k-1)$. In this context, Theorem \ref{mainlasttheorem} may be reformulated as follows.
\end{example}

\begin{proposition}\label{proposition2}
The exact analytical solution of the delta fractional delayed difference system has the following form
   \begin{align}\label{totalsolution}
    z(k)&={\mathbb{D}_{\alpha,\alpha,r}^{M,N}}(k-1)\phi(1-r)
    +\int_{-r}^{-1}{\mathbb{D}_{\alpha,\alpha,r}^{M,N}}\left(k-1-r-s\right) N\phi(s) \Delta s \nonumber\\
     & +\int_{1}^{k} {\mathbb{D}_{\alpha,\alpha,r}^{M,N}}\left(k-r-s\right)\daleth(s-1)\Delta s.
  \end{align}
\end{proposition}
\begin{remark}
Here are some of our findings.
\begin{enumerate}
\item Proposition \ref{proposition2} is novel for the delta fractional delayed difference system with commutative coefficient matrices.
  \item  Theorem \ref{mainlasttheorem}  matches up with \cite[Theorem 3.6]{i8}, for $\alpha=\beta=1$, $r=h+1$,
  \item Theorem \ref{mainlasttheorem} coincides  with \cite[Theorem 3.5]{iek7}, for $\alpha=\beta=1$, $r=h+1$, and $MN=NM$.
\end{enumerate}
\end{remark}

\section{Conclusion}
  In the current paper, we introduce the nabla Riemann Liouville fractional retarded difference system with the noncommutative coefficients. In the sequel, we novelly launch out the discrete delayed perturbation of the nabla M-L type matrix function, investigate its properties, determine some relations between it and the available ones related to it, and graphically compare with some of the present ones. We investigate solutions of homogeneous and nonhomogeneous versions of the RL fractional delayed difference system brick by brick and present their solutions based on the discrete retarded perturbation of the M-L function, and lastly discuss some special cases.

  There are a lot of studies whose examples can be seen in references cited in the next sentence to be done  on this paper. It can be questionable whether the RL fractional retarded difference system given in (\ref{solution2}) is   the relative controllable\cite{c1} \cite{c2}\cite{c3}\cite{c4}, iterative learning controllable\cite{c5}\cite{c6}, exponential stable\cite{c11}\cite{c12}, asymptotic stable\cite{c13}\cite{c14}\cite{c15}, finite-time stable\cite{c16}, etc. Our system can be extended to some different versions as higher-order linear discrete version\cite{c7}\cite{c8} \cite{c9},  multi-retarded version\cite{c10}, and variable coefficients' version, and all of just aforementioned properties can be examined for each of these new systems.

\vspace{0.5cm}
\textbf{The conflict of interest statement:} The authors declare that they have no conflicts of interest to state(declare).

\end{document}